\documentclass{amsart}
\usepackage{amssymb}%
\usepackage{graphicx}%
\usepackage{amsthm}%
\usepackage{amsmath}%
\usepackage[normalem]{ulem}
\usepackage{amsfonts} 
\usepackage{latexsym}%
\usepackage{epsfig}%
\usepackage{amscd}%
\usepackage{color}%
\usepackage[linktoc=all, pagebackref, hyperindex]{hyperref}%
\hypersetup{colorlinks,  citecolor=blue,  filecolor=blue,  linkcolor=red,  urlcolor=black}

\theoremstyle{plain}
  \newtheorem{thm}{Theorem}[section]
  \newtheorem{prop}[thm]{Proposition}
  \newtheorem{lem}[thm]{Lemma}
  \newtheorem{cor}[thm]{Corollary}

\theoremstyle{definition}
  \newtheorem{dfn}[thm]{Definition}
  \newtheorem{exmp}[thm]{Example}

  \newtheorem{rem}[thm]{Remark}

    {\begin{list}
        {\noindent\makebox[0mm][r]{\rm(\roman{enumi})}}
        {\leftmargin=5.5ex \usecounter{enumi}}
    }
    {\end{list}}

\newcommand{\codim}{\operatorname{codim}}

\newcommand{\depth}{\operatorname{depth}}

\newcommand{\reg}{\operatorname{reg}}

\newcommand{\gr}{\operatorname{gr}}

\newcommand{\ini}{\operatorname{in}}

\newcommand{\G}{\operatorname{Gr}}

\newcommand{\Max}{\operatorname{Max}}

\renewcommand{\reg}{\operatorname{reg}}

\newcommand{\AP}{\operatorname{Ap}}
\newcommand{\Ap}{\operatorname{Ap}}

\newcommand{\co}{\operatorname{cone}}

\newcommand{\ord}{\operatorname{ord}}
\newcommand{\re}{\operatorname{rem}}

\def\RR{{\mathbb R}}

\def\PP{{\mathbb P}}
\def\KK{{\mathbb K}}
\def\QQ{{\mathbb Q}}
\def\NN{{\mathbb N}}
\def\ZZ{{\mathbb Z}}

\def\DD{{\mathbb D}}

\def\kk{\Bbbk}
\def\0{{\mathbf 0}}
\def\c{{\mathbf c}}

\def\1{{\mathbf 1}}
\def\a{{\mathbf a}}
\def\b{{\mathbf b}}
\def\z{{\mathbf z}}

\def\x{{\mathbf x}}

\def\u{{\mathbf u}}

\def\w{{\mathbf w}}

\def\fm{{\mathfrak m}}
\def\fn{{\mathfrak n}}

\def\fq{{\mathfrak q}}

\def\<{{\langle}}
\def\>{{\rangle}}

\newcommand{\excise}[1]{}
\begin{document}

\title[Monomial minimal reduction ideals]{Simplicial affine semigroups with monomial minimal reduction ideals}

\author{Marco D'Anna,  Raheleh Jafari,  Francesco Strazzanti}

\address{Marco D'Anna, 	Dipartimento di Matematica e Informatica, Universit\`a degli Studi di Catania, Viale Andrea Doria 6, 95125 Catania, Italy.}
\email{mdanna@dmi.unict.it}

\address{Raheleh Jafari,  Mosaheb Institute of Mathematics, Kharazmi University, and School of Mathematics, Institute for Research in Fundamental Sciences (IPM), P.O. Box 19395-5746, Tehran, Iran.} \email{rjafari@ipm.ir}

\address{Francesco Strazzanti, Dipartimento di Matematica, Universit\`a degli Studi di Torino, Via Carlo Alberto 10, 10123 Torino, Italy.}
	\email{francesco.strazzanti@gmail.com}



\begin{abstract}
We characterize when the monomial maximal ideal of a simplicial affine semigroup ring has a monomial minimal reduction. When this is the case, we study the Cohen-Macaulay  and Gorenstein properties of the associated graded ring and  provide several bounds for the reduction number with respect to the monomial minimal reduction.	
\end{abstract}

\maketitle

\section*{Introduction}

In this paper we study affine semigroup rings $\KK[S]$, with $S$ fully embedded in $\mathbb N^d$, i.e. subalgebras  of the polynomial ring $\KK[x_1,\dots,x_d]$ generated by the monomials with exponents in $S$,  where $\KK$ is a field.  The semigroup  ring $R=\KK[S]$ is of Krull dimension $d$ with a unique monomial maximal  ideal $\fm$.
The study of affine semigroups and corresponding semigroup rings has many applications in different areas of mathematics; for instance, it gives the combinatorical background for  the theory of toric varieties. 
Many properties of affine semigroup rings have been studied in connection to the properties of 
the associated semigroups (see, e.g., \cite{BH} and \cite{BGT}).
A case of particular interest is when $S$ is a simplicial affine semigroup,  i.e.,  when the cone  over $S$, $\co(S)$, has $d$ extremal rays. Under this hypothesis, denoting by $\a_1,\dots,\a_d$ the componentwise smallest nonzero vectors of $S$ situated respectively, on each extremal ray of $\co(S)$ (which are called the extremal rays of $S$), many properties of 
$\KK[S]$ (like, e.g., Cohen-Macaulayness or Gorensteinness) can be studied via the Ap\'ery set of $S$  with respect to $\{\a_1,\dots,\a_d\}$ (see \cite{RG-1998}). 

Affine simplicial semigroups  provide also a natural generalization of numerical semigroups, whose theory has been developed mainly in connection to the study of curve singularities.  In the case of numerical semigroup rings, a
remarkable property is that the  monomial maximal ideal always has a monomial minimal reduction, which is  the principal ideal
generated by the monomial of minimal  positive degree.
This is a key fact to explore the connection between the associated graded ring $\gr_\fm(R)=\oplus_{i=0}^\infty \, \fm^i/\fm^{i+1}$, the ring $R$, and the associated semigroup.
However,  in the case of affine simplicial semigroup rings, a monomial minimal reduction for $\fm$ in general does not exist; so it is natural to study when it exists and, if this is the case, to explore which information can be deduced from it. 

We firstly observe that, if such a minimal reduction exists, it has to be generated by the monomials  $\x^{\a_1},\dots,\x^{\a_d}$ corresponding to the extremal rays,  provided that $\KK$ is infinite. Moreover, setting
$I=(\x^{\a_1},\dots,\x^{\a_d})$ and denoting the multiplicity of $I$ by $e(I)$, we show that $\KK[S]$ is Cohen-Macaulay if and only if $e(I)$ equals the cardinality of the Ap\'{e}ry set of $S$ with respect to the set  $\{\a_1, \dots, \a_d\}$. Since $e(I)=e(\fm)$ precisely when  $I$ is a  reduction ideal of $\fm$, if $I$ is a (minimal) reduction ideal of $\fm$, it is possible to interpret the multiplicity of $R$ in the context of the associated semigroup. 
  We characterize the existence of the monomial minimal reduction of $\fm$ by looking at the degrees of those generators of the semigroup that are not among the extremal rays of $S$, in  Theorem~\ref{monomial red}. Meanwhile, we also provide a bound for the reduction number in terms of rational coordinates of the generators of $S$.  
	We deepen the Cohen-Macaulay case, showing that, when $R$  has minimal multiplicity, the monomial minimal reduction exists  precisely when the set of non-zero Ap\'{e}ry elements  of $S$ with respect to the extremal rays coincide with the minimal generating set, in Proposition~\ref{min red and min mult}. Applying these results to the case of affine semigroups in dimension~2, we provide a characterization of the existence of a monomial minimal reduction of $\fm$ (Proposition \ref{mon red dim 2}), an explicit formula for $R$ to be Cohen-Macaulay (Proposition \ref{parallelogram}), 
	and a formula for the multiplicity of $R$ in terms of some Ap\'{e}ry sets of $S$ (Proposition~\ref{3.8}).

	In the case that $I=(\x^{\a_1},\dots,\x^{\a_d})$ is a minimal reduction for $\fm$, two problems naturally arise: if it is possible to characterize the Cohen-Macaulay and Gorenstein properties for the associated graded ring, and if it is possible to determine  a bound for $r_I(\fm)$, the reduction number of $\fm$ with respect to $I$. As for the first problem, we give a complete answer in Theorem~\ref{thm:gr} and Proposition~\ref{4.7}. We also show that  $r_I(\fm)$ equals the maximum order of Ap\'{e}ry elements with respect to $\{\a_1,\dots,\a_d\}$, when the associated graded ring is Cohen-Macaulay. 	
	The second problem is also motivated by the bound given  in \cite{Hoa-Stuckrad}, in the particular case of homogeneous affine simplicial semigroups. It is known that $r_I(\fm)\leq e(R)$, \cite{Trung-1987}. If $\KK$ has zero characteristic, in \cite{Vasconcelos-1996} Vasconcelos shows that $r_I(\fm)\leq e(R)-1$. For homogeneous affine simplicial semigroups, $I$ is a reduction ideal of $\fm$ and in \cite{Hoa-Stuckrad} it is proved that $r_I(\fm)\leq e(R)-\codim(R)$. We confirm the same bound for simplicial affine semigroups with monomial minimal reduction and Cohen-Macaulay associated graded ring, in Proposition~\ref{HS}. As Example~\ref{ex:4.7} shows, the Cohen-Macaulay condition of $\gr_\fm(R)$ can not be removed in this result. Finally we provide a formula for the Castelnouvo-Mumford regularity of $\gr_\fm(R)$, when $\fm$ has monomial minimal reduction and $\gr_\fm(R)$ is Cohen-Macaulay, in Proposition~\ref{p}.

\section{Preliminaries}
 We start by recalling some definitions and known results that we will use in the paper. For more details see \cite{RG-1999, BH}.
Let $S$ be an affine semigroup fully embedded in $\NN^d$, where $\NN$ denotes the set of non-negative integers and $d\in\NN\setminus\{0\}$. Then, $S$ has a unique minimal generating set $\{\a_1,\dots,\a_n\}$. Let
 $$
 \co(S)=\left \{ \sum_{i=1}^n \lambda_i \a_i: \lambda_i \in \RR_{\geq 0}, \text{ for } i=1,\dots, n \right\},
 $$
 denote the rational polyhedral cone generated by $S$.
 The {\em dimension} (or rank) of $S$ is defined as the dimension of  the affine subspace it generates, which is the same as the dimension of the subspace generated by  $\co(S)$. As $S$ is fully embedded in $\NN^d$, its dimension is $d$.   
 The $\co(S)$ is the intersection of finitely many closed linear half-spaces in $\RR^d$, each of whose bounding hyperplanes contains the origin. These half-spaces are called {\em support hyperplanes}. 
 For $d=1$, $\co(S)=\RR_{\geq 0}$, and for $d=2$ the support hyperplanes are one-dimensional vector spaces  which are called the {\em extremal rays} of $\co(S)$. When $d>2$, 
  the intersection of any two adjacent support hyperplane is a one-dimensional vector space, which is said to be an {\em extremal ray} of $\co(S)$. An element of $S$ is called {\em extremal ray of $S$} if it is the smallest non-zero vector of $S$ in an extremal ray of $\co(S)$.  The $\co(S)$ has at least $d$  extremal rays. $S$ is called {\em simplicial} when $\co(S)$  has exactly $d$ extremal rays.   Throughout the paper, $S$ is a simplicial affine semigroup and $E=\{\a_1,\dots,\a_d\}$ denotes the set of extremal rays of $S$. Then, $\a_1,\ldots,\a_d$ are $\QQ$-linearly independent and 
 $$S\subseteq\sum^d_{i=1}\QQ_{\geq0}\a_i.$$

Throughout the paper $S$ is a simplicial affine semigroup minimally generated by $\a_1,\dots,\a_{d+s}$, where $\a_1,\dots,\a_d$ are the extremal rays of $S$.  
Given $\a\in \NN^d\cap\co(S)$, there exist unique rational numbers $l_i\in\QQ_{\geq 0}$ such that $\a=\sum^d_{i=1}l_i\a_i$.  We denote  $\deg(\a)=\sum^d_{i=1}l_i$.
  Note that  $\deg(\a)+\deg(\b)=\deg(\a+\b)$ for all  $\a,\b\in \NN^d \cap \co(S)$. The hyperplane defined by $\a_1,\dots,\a_d$ contains an integral vector $\a\in \co(S)$ precisely when $\deg(\a)=1$. The set of integral vectors of the convex hull of the extremal rays together with the origin of coordinates consists of the integral vectors $\a\in\co(S)$ with $\deg(\a)\leq1$.    

For a field  $\KK$, the {\it semigroup ring} $\KK[S]$ is the $\KK$-subalgebra of the polynomial ring $\KK[x_1,\dots,x_d]$ generated by the monomials with exponents in $S$.  
 Given $\a=(a_1, \dots, a_d) \in \NN^d$, we set $\x^{\a}=x_1^{a_1}x_2^{a_2}\cdots x_d^{a_d}$.

All over  the paper, $R=\KK[S]$ denote the semigroup ring $\KK[\x^{\a_1},\dots,\x^{\a_{d+s}}]$. Consider the standard grading on the polynomial ring, where  the degree of $x_i$ is the vector $(0,0,\dots, 0,1,0,\dots,0)$ with size $d$ and a unique $1$ in the $i$-th entry.  Then, $R$ is a graded ring with the graded maximal ideal $\fm=(\x^{\a_1},\dots,\x^{\a_{d+s}})$.

  A subset $H\subseteq S$ is called an ideal of $S$, if $S+H\subseteq H$. The largest non-trivial ideal of $S$ is the maximal ideal $M=S\setminus\{0\}$.     
For any subset $H$ of $S$, let $\KK\{H\}$ denote the $\KK$-linear span of the monomials $\x^\a$ with $\a\in H$. Then, $I$ is a monomial ideal of $\KK[S]$ if and only if  $I=\KK\{H\}$ for some ideal $H$ of $S$. Note that $\fm=\KK\{M\}$.

Let $H\subseteq F$ be ideals of $S$. 
By \cite[Lemma~1.9]{GSW-76}, $F\setminus H$ is a finite set if and only if $\lambda_R(\KK\{F\}/\KK\{H\})<\infty$, where $\lambda_R(-)$ denotes the length of an $R$-module. 
  In this case, 
  \begin{equation}\label{lambda}
  |F\setminus H|=\dim_\KK(\KK\{F\}/\KK\{H\})=\lambda_R(\KK\{F\}/\KK\{H\}). 
  \end{equation}

    Note that if $H_1$ and $H_2$ are ideals of $S$, then $H_1+H_2=\{h_1+h_2 \, ; \, h_1\in H_1 ,h_2\in H_2\}$ is
 also an  ideal of $S$. For a positive integer $n$ and an ideal $H$, by $nH$ we mean $H+(n-1)H$ and $2H=H+H$.

   For $\a\in S$, 
 the maximum integer $n$ such that $nM$ contains $\a$ is called the order of $\a$ in $S$ and it is denoted by $\ord_S(\a)$. In other words, $\a\in nM\setminus(n+1)M$ if
 and only if $n=\ord_S(\a)$. So that $\a$ may be written as $\a=\sum^{d+s}_{i=1}r_i\a_i$ for some non-negative integers $r_i$ such  that $\sum^{d+s}_{i=1}r_i=n$.  We call this representation a {\em maximal expression} of $\a$. 
 
Let $I_S$ denote the kernel of the $\KK$-algebra homomorphism $\varphi:\KK[z_1,\dots,z_{d+s}]\mapsto\KK[S]$, defined by $z_i\mapsto\x^{\a_i}$, for $i=1,\dots,d+s$. Then, the {\em defining ideal} $I_S$ is a binomial prime ideal,  \cite[Proposition~1.4]{Herzog-1970}.

  \begin{rem}\label{rem:deg1}
  	It is easy to see that	$\deg(\a_{d+j})=1$ for $j=1,\dots,s$ if and only if all expressions of any element of $S$ are maximal. In this case $I_S$ is a homogeneous ideal and $R\cong\gr_\fm(R)$.
  \end{rem}

The {\em Ap\'{e}ry set} of $S$ with respect to an element $\b\in S$ is defined as $\Ap(S,\b)=\{\a\in S \, ; \, \a-\b\notin S\}$. Since $S\subseteq\NN^d$, for $\b\neq \0$ we have  $\0\in \Ap(S,\b)$, where $\0$ denotes the zero vector of $\NN^d$.   
 We set 
$$\Ap(S,E)=\{\a\in S \, ; \, \a-\a_i\notin S, \text{ for  } i=1,\dots,d\}=\cap^d_{i=1}\Ap(S,\a_i).$$   Moreover, let $\G(S)=\{\a-\b \, ; \, \a,\b\in S\}$ be the {\em group of differences} of $S$,  and let 
\[
P_S=\left\{ \z \in\NN^d \, ; \, \z=\sum_{i=1}^d \lambda_i \a_i \text { with } 0\leq \lambda_i <1 \text{ for } i=1,\dots,d\right\}.
\]
Any element  $\a\in \NN^d\cap\co(S)$ is uniquely written as $\a=\sum_{i=1}^d n_i \a_i+\re(\a)$, where $\re(\a)\in P_S\cap\NN^d$ and $n_1,\dots,n_d\in\NN$. 

 The following lemma is an easy consequence of \cite[Lemma~1.4 and Corollary~1.6]{RG-1998}. For the convenience of the reader, we include this fact here.

\begin{lem}\label{lemAp}
	The following statements hold: 
	\begin{enumerate}
		\item[(a)] $\re(\AP(S,E))=P_S\cap\G(S)$; 
		\item[(b)] $R$ is Cohen-Macaulay if and only if  the restriction of $\re(-)$ to $\AP(S,E)$ is an injective map.
	\end{enumerate}
\end{lem}
\begin{proof}
	(a) If $\z=\re(\a)$ for some $\a\in\AP(S,E)$, then $\z=\a-\sum^d_{i=1}n_i\a_i$ for some $n_i\geq 0$, which implies $\z\in\G(S)$. Now, let $\z\in P_S\cap\G(S)$. Then, $\z+\b\in S$ for some $\b\in S$. Since $S$ is simplicial, $n\b=\sum^d_{i=1}r_i\a_i$ for some $r_i \in \mathbb{N}$ and $n \gg 0$. Consequently, $\z+n\b=\z+\sum^d_{i=1}r_i\a_i\in S$. 
	Thus $\z+\sum^d_{i=1}r_i\a_i=\w+\sum^d_{i=1}s_i\a_i$ for some $\w\in\Ap(S,E)$ and $s_1,\dots,s_d\in\NN$. As $\z\in P_S$, we derive $\z=\re(\w)$.
	
	(b) Let  $\w_1,\w_2\in\AP(S,E)$ and $\w_j=\sum^d_{i=1}k_{i_j}\a_i+\re(\w_j)$ for $k_{i_j}\in\NN$ and $j=1,2$. Then, 
	\[
	\w_1-\w_2=\re(\w_1)-\re(\w_2)+\sum^d_{i=1}(k_{i_1}-k_{i_2})\a_i.
	\]
	Note that $\sum^d_{i=1}(k_{i_1}-k_{i_2})\a_i\in\G(\a_1,\dots,\a_d)$ and $\re(\w_1),\re(\w_2)\in P_S$. Thus, $\w_1-\w_2\in\G(\a_1,\dots,\a_d)$ if and only if $\re(\w_1)=\re(\w_2)$. The result follows by  \cite[Corollary~1.6]{RG-1998}. 
\end{proof}

\begin{dfn}
An ideal $J\subseteq\fm$ is called a {\em reduction ideal} of $\fm$ if $\fm^{n+1}=J\fm^n$ for $n\gg0$. 
A reduction ideal is said to be a {\em minimal reduction ideal} if it is minimal with respect to set inclusion.
  If $J$ is a minimal reduction ideal of $\fm$, the smallest positive integer $r$ such that $\fm^{n+1}=J\fm^n$ for $n\geq r$ is called {\em reduction number} of $\fm$ with respect to $J$ and it is denoted by $r_J(R)$.  
\end{dfn}

    \begin{lem}\label{mm}
    	The following statements hold.
    	\begin{enumerate}
    		\item If $(\x^{\a_1},\dots,\x^{\a_d})$ is a reduction ideal of $\fm$, then it is a minimal reduction.
    		\item If $\KK$ is  an infinite field, then  the unique possible monomial minimal reduction ideal of $\fm$ is the ideal $(\x^{\a_1},\dots,\x^{\a_d})$. 
    	\end{enumerate}
  \end{lem}
\begin{proof}
	Let $I$ be a monomial reduction ideal of $\fm$. Since $I$ is a monomial ideal and $\fm$ is the unique monomial maximal   ideal, it is easy to verify that  $I$ and $IR_\fm$ have the same number of generators, and
	$IR_\fm$ is also a reduction ideal of $\fm R_\fm$.
Thus, (1) follows by \cite[Corollary~8.3.6]{Huneke-Swanson}.
 As    $\a_1,\dots,\a_d$ are linearly independent, $\x^{\a_1},\dots,\x^{\a_d}$ should belong to the generating set of $I$. If $\KK$ is infinite, then  \cite[Proposition~8.3.7]{Huneke-Swanson} implies that  $IR_\fm$ is generated by $d$ elements. 	 Therefore,  $I=(\x^{\a_1},\dots,\x^{\a_d})$.
\end{proof}

 For  an $\fm$-primary ideal $\fq$, let $e(\fq)$ denote the {\em multiplicity} of $\fq$, i.e.   
   \[
   e(\fq)=\lim_{n\rightarrow\infty}\frac{d!}{n^d}\lambda_R(R/\fq^n).
   \]
We refer to $e(\fm)$  as the {\em multiplicity of $R$} and write $e(R)$ for it.

\begin{lem}\label{eqm}
	Let $\fq$ be a parameter ideal of $\fm$. 
	\begin{enumerate}
		\item[(a)] $e(\fq)=e(R)$ if and only if $\fq$ is a  reduction ideal of $\fm$.
		\item[(b)]  $R$ is Cohen-Macaulay if and only if $e((\x^{\a_1},\dots,\x^{\a_d}))=|\AP(S,E)|$.
		\item[(c)] $e(R)\leq|\Ap(S,E)|$.	
	\end{enumerate}
	\end{lem} 
\begin{proof}
(a)  is a  consequence of \cite[Theorem~11.3.1]{Huneke-Swanson} and \cite[Corollary~1.2.5]{Huneke-Swanson}. 
Note that $\lambda_R(R/(\x^{\a_1},\dots,\x^{\a_d}))=|S\setminus H|=|\Ap(S,E)|$, where $H=\{\a_1,\dots,\a_d\}+S$. Thus, (b) follows by \cite[Theorem~17.11]{Matsumura}. 

Let now $\fq=(\x^{\a_1},\dots,\x^{\a_d})$. Since $S$ is simplicial, $\fq$ provides a  parameter ideal for $R$. As  $\lambda_R(R/\fq)=|\Ap(S,E)|$, the result follows by  \cite[Theorem~14.10]{Matsumura}.
\end{proof}

\begin{rem}  The argument in the proof of \cite[Corollary 4.7]{Ojeda-Tenorio-2017}  provides another perspective of Lemma~\ref{eqm}(b). It is shown  that  $R$ is Cohen-Macaulay if and only if it is a free $\KK[z_1,\dots,z_d]$-module of rank $|\AP(S,E)|$.  Therefore, if $R$ is Cohen-Macaulay, then  $e(\x^{\a_1},\dots,\x^{\a_d})=|\Ap(S,E)|e(\KK[z_1,\dots,z_d])=|\Ap(S,E)|$. 
\end{rem}

\begin{prop}\label{prop:mon}
If two of the following conditions hold, then the third one is also satisfied. 
\begin{enumerate}
	\item[(a)] $e(R)=|\AP(S,E)|$.
	\item[(b)]$(\x^{\a_1},\dots,\x^{\a_d})$ is a minimal reduction ideal of $\fm$. 
		\item[(c)] $R$  is  a Cohen-Macaulay ring.
\end{enumerate}
\end{prop}
\begin{proof}
	Let $\fq=(\x^{\a_1},\dots,\x^{\a_d})$. 
If $\fq$ is a minimal reduction ideal of $\fm$, then   $e(R)=e(\fq)$, by Lemma~\ref{eqm}. Thus, (a) and (c) are equivalent. 
If  $R$ is Cohen-Macaulay, then the equivalence of (a) and (b) follows again by Lemma~\ref{eqm}. 
\end{proof}

\begin{exmp} \label{CMnotMonomial}
Let $S$ be the affine semigroup generated by $\a_1=(2,0)$, $\a_2=(0,4)$, $\a_3=(1,1)$, and $\a_4=(1,2)$. It is possible to see that $R=\mathbb{Q}[\x^{\a_1},\dots,\x^{\a_4}]$ is Cohen-Macaulay and $e(R)=6$, for instance by using Macaulay2~\cite{Mac2}. On the other hand, in this case $\AP(S,E)=\{(0,0),(1,1),(2,2),(3,3),(1,2),(2,3),(3,4){\, (4,5)}\}$ and, then, $e(R) < |\AP(S,E)|$. By the previous proposition, this also implies that $(\x^{\a_1},\x^{\a_2})=(x_1^2,x_2^4)$ is not a minimal reduction of the maximal ideal of $R$. 
\end{exmp}

Assuming $I$ is a minimal reduction ideal of $\fm$ we can deduce an interesting consequence.

\begin{rem}\label{rem:red}
	Let $I=(\x^{\a_1},\dots,\x^{\a_d})$ be a  reduction ideal of $\fm$. Then, given $\z\in S$ with $\ord_S(\z)\geq r_I(R)+1$, the monomial $\x^\z$ is divided by a monomial $\x^{\a_i}$ for some $1\leq i\leq d$. Therefore, 
$r_I(R)\geq\max\{\ord_S(\w) \, ; \, \w\in\Ap(S,E)\}$.
\end{rem}

\section{Reduction  ideals of the maximal ideal}
 Let $c_j$ denote the smallest positive integer with $c_j\a_{d+j}\in\sum^d_{i=1}\NN \a_i$ for $j=1,\dots,s$.
Then, for every such $j$, there are unique integers  $l_{j_1},\dots,l_{j_d}\in\NN$ such that 
\[
c_j\a_{d+j}=\sum^d_{i=1}l_{j_i}\a_i.
\]
 Note that $\deg(\a_{d+j})=e_j/c_j$, where $e_j=l_{j_1}+\dots+l_{j_d}$.

\begin{lem}\label{1-1}
If $\deg(\a_{d+j})\geq1$ for some $1\leq j\leq s$, then 
$$\x^{n\a_{d+j}}\in(\x^{\a_1},\dots,\x^{\a_d})\fm^{n-1}$$
for $n\geq e_j$.
\end{lem}
\begin{proof}
Let $I=(\x^{\a_1},\dots,\x^{\a_d})$. For every $n \geq e_j \geq c_j$,
\[
\x^{n\a_{d+j}}=\x^{c_j\a_{d+j}}\x^{(n-c_j)\a_{d+j}}=(\x^{l_{j_1}\a_1}
\cdots\x^{l_{j_d}\a_d})\x^{(n-c_j)\a_{d+j}}\in I\fm^{t}, 
\]
where $t=\sum^d_{i=1}l_{i_j}-1+n-c_j\geq n-1$. It follows that $\x^{n\a_{d+j}}\in I\fm^{n-1}$.
\end{proof}
	
 In the following theorem we characterize when there exists a monomial minimal reduction ideal of $\fm$.

  \begin{thm}\label{monomial red}
  	The following statements are equivalent:
  	\begin{enumerate}
  		\item[(a)] $\deg(\a_{d+j})\geq 1$, for every $j=1,\dots,s$;
  		\item[(b)]  The ideal $I=(\x^{\a_1},\dots,\x^{\a_d})$ is a  reduction of $\fm$;
  	  	\end{enumerate} 
  	If the above statements hold, the reduction number of $\fm$ with respect to $I$ is at most $sl-1$, where $l=\max\{e_1,\dots,e_s\}$.  
  \end{thm} 
  \begin{proof}
  	
  	(a)$\Rightarrow$(b)  We may assume that $s\geq 1$, otherwise the claim is trivial. Let $n\geq\max\{e_1,\dots,e_s\}$. Then, $\x^{n\a_{d+j}}\in I\fm^{n-1}$ by Lemma~\ref{1-1}.  
  	For every monomial $\u=\x^{n_1\a_{d+1}}\cdots\x^{n_s\a_{d+s}}\in\fm^{sn}$ there is at least one index $j$ such that $n_j \geq n$. Therefore, $\u\in \x^{n\a_{d+j}}\fm^{(s-1)n}\subseteq I\fm^{n-1}\fm^{(s-1)n} \subseteq I\fm^{sn-1}$ and this is enough to show that $\fm^{rn} \subseteq I\fm^{rn-1}$.
 Hence, $I$ is a reduction ideal of $\fm$.
This also proves the last statement, when (a) holds.
  	
  	(b)$\Rightarrow$(a)
Fix $t$ such that $\deg(\a_{d+t})=\min\{\deg(\a_{d+j}) \, ; \, j=1,\dots,s\}$. If $n$ is large enough, (b) implies that
  	\[
  	\x^{(n+1)\a_{d+t}}=\x^{\a_{j}}\x^{\u_1+\dots+\u_n},
  	\] for some $1\leq j\leq d$ and $\u_1,\dots,\u_n\in S$. Consequently, 
  	\[
  	(n+1)\deg(\a_{d+t})=1+\sum^n_{i=1}\deg(\u_i)\geq 1+n\min\{1,\deg(\a_{d+t})\},
  	\]
  	which implies $\deg(\a_{d+t})\geq 1$.
  \end{proof}
  
\begin{rem}\label{D}
 Denote by $\DD$ the hyperplane passing through the extremal rays $\a_1,\dots,\a_d$, and by $\PP$ the convex hull of the  set $\{\0,\a_1,\dots,\a_d\}$. As a geometrical interpretation, Theorem~\ref{monomial red} states that   $(\x^{\a_1},\dots,\x^{\a_d})$ is a monomial reduction ideal  of $\fm$ precisely when  $\PP\cap\{\a_{d+1},\dots,\a_{d+s}\}=\DD\cap\{\a_{d+1},\dots,\a_{d+s}\}$. In other words, $\a_i$ is located on or above $\DD$, for $i=d+1,\dots,d+s$.  	
\end{rem}

The following example shows that the bound for the reduction number found in Theorem~\ref{monomial red} is sharp.

 \begin{exmp}\label{exm-mon-red}
Let $S \subseteq \mathbb{N}^2$ be the affine semigroup generated by $\a_1=(3,1),\a_2=(0,4),\a_3=(2,2)$, where  $\a_1$ and $\a_2$ are the extremal rays.  Then, $\a_3=2/3\a_1+1/3\a_2$ and $\deg(\a_3)=1$.
Therefore, $(\x^{\a_1},\x^{\a_2})$ is a minimal reduction ideal of $\fm=(\x^{\a_1},\x^{\a_2},\x^{\a_3})$.  
Moreover, in this case $e_3=3$ and the reduction number is $2$, which confirms the sharpness of the bound given in Theorem~\ref{monomial red}.
  \end{exmp}

 If $R$ is Cohen-Macaulay, then $e(R)\geq {\rm edim}(R)-\dim(R)+1= d+\codim(R)-d+1=\codim(R)+1$, where ${\rm edim}(R)$ is the embedding dimension of $R$ and $\codim(R)=d-{\rm edim}(R)=s$, see \cite{Abhyankar-1967}. If the equality holds, then $R$ is said to be a Cohen-Macaulay ring of minimal multiplicity.

 \begin{prop}\label{min red and min mult}
 If $R$ is Cohen-Macaulay, the following statements are equivalent:
 	\begin{enumerate}
 		\item[(a)] $\Ap(S,E)$ has $1+\codim(R)$ elements; 
 		\item[(b)]  $(\x^{\a_1},\dots,\x^{\a_d})$ is a reduction ideal of $\fm$ and $R$ is of minimal multiplicity, i.e., $e(R)=1+\codim(R)$;
 		\item[(c)] $(\x^{\a_1},\dots,\x^{\a_d})\fm=\fm^2$.
 	\end{enumerate} 
 \end{prop}
 \begin{proof}
 	(a) $\Rightarrow$ (b) By assumption, we have $\AP(S,E)=\{\0,\a_{d+1},\dots,\a_{d+s}\}$. Let $\a_{d+t}$ be such that $\deg(\a_{d+t})=\min\{\deg(\a_{d+j}) \, ; \, 1\leq j\leq s\}$. Since $2\a_{d+t}\notin\cap_{i=1}^{d}\AP(S,\a_i)$, we have $2\a_{d+t}=\a_i+\b$, for some $\b\in S$ and some $i$. Then, $2\deg(\a_{d+t})=1+\deg(\b)\geq 1+\min\{1,\deg(\a_{d+t})\}$, which implies $\deg(\a_{d+t})\geq 1$. Therefore, $(\x^{\a_1},\dots,\x^{\a_d})$ is a minimal reduction ideal of $\fm$, by Proposition~\ref{monomial red}. The equality $e(R)=1+\codim(R)$ follows by Proposition~\ref{prop:mon}. 	
 	
 	(b) $\Rightarrow$ (c) By Lemma~\ref{mm},  the ideal $(\x^{\a_1},\dots,\x^{\a_d})$ is a minimal reduction ideal of $\fm$. Since $e(R)=|\Ap(S,E)|$ by Proposition~\ref{prop:mon}, we have $|\Ap(S,E)|=1+\codim(R)$, which is $\Ap(S,E)=\{\0,\a_{d+1},\dots,\a_{d+s}\}$. In particular, for any $\x^\a\in\fm^2$, we have $\a\notin\Ap(S,E)$,  which means $\a=\a_i+\b$ for $1 \geq i \geq d$ and some $\b \in S \setminus \{0\}$. Therefore, $\x^\a\in(\x^{\a_1},\dots,\x^{\a_d})\mathfrak{m}$.
 	
 	(c) $\Rightarrow$ (a) Clearly, $\AP(S,E)=\{\0,\a_{d+1},\dots,\a_{d+s}\}$.
 \end{proof}

In Theorem \ref{monomial red} we have seen that if $\deg(\a_{d+j}) \geq 1$ for every $j=1, \dots, s$, then $(x^{\a_1}, \dots, \x^{\a_{d}})$ is a minimal reduction ideal of $\fm$. On the other hand, if $\deg(\a_{d+j}) < 1$ for $j=1, \dots, i$ and $\deg(\a_{d+j}) \geq 1$ for $j=i, \dots, s$, the same proof shows that $J=(x^{\a_1}, \dots, \x^{\a_{d+i}})$ is a reduction ideal of $\fm$, but it is never minimal if $i\geq 1$. We conclude this section by providing a better reduction ideal in this case, indeed the ideal that we find is always contained in $J$.

\begin{lem}\label{majic}
	If $1\leq j\leq s$, then $\x^{e_j\a_i}\in(\x^{\a_1}-\x^{\a_2},\dots,\x^{\a_1}-\x^{\a_d},\x^{\a_{d+j}})\fm^{r_j-1}$ for every $1\leq i\leq d$, where $r_j=\min\{e_j,c_j\}$.
\end{lem}
\begin{proof}
	Let $I=(\x^{\a_1}-\x^{\a_2},\dots,\x^{\a_1}-\x^{\a_d},\x^{\a_{d+j}})$. Since $\x^{e_j\a_1}-\x^{e_j\a_i}\in I\fm^{e_j-1}$ for every $i$, it is enough to show that $\x^{e_j\a_1}\in I\fm^{r_j-1}$.
	We recall that $c_j\a_{d+j}=\sum^d_{i=1}l_{j_i}\a_i$ and $e_j=\sum^d_{i=1}l_{j_i}$. We may write
	\begin{eqnarray*}
		\x^{e_j\a_1} &=& \x^{(e_j-l_{j_2})\a_1}(\x^{l_{j_2}\a_1}-\x^{l_{j_2}\a_2})  +\x^{(e_j-l_{j_2})\a_1}\x^{l_{j_2}\a_2}= \\
		&=&  \x^{(e_j-l_{j_2})\a_1}(\x^{l_{j_2}\a_1}-\x^{l_{j_2}\a_2})      + 
		\x^{(e_j-l_{j_2}-l_{j_3})\a_1}\x^{l_{j_2}\a_2}
		(\x^{l_{j_3}\a_1}-\x^{l_{j_3}\a_3})+ \\
		& & + \ \x^{(e_j-l_{j_2}-l_{j_3})\a_1}\x^{l_{j_2}\a_2}\x^{l_{j_3}\a_3}=
		\\
		&=& \sum_{i=2}^d\x^{(e_j-\sum_{t=2}^il_{j_t})\a_1}\x^{l_{j_2}\a_2}\cdots\x^{l_{j_{i-1}}\a_{i-1}}(		\x^{l_{j_i}\a_1}-\x^{l_{j_i}\a_i}) + \x^{l_{j_1}\a_1} \cdots \x^{l_{j_d}\a_d},
	\end{eqnarray*}
	which easily implies the thesis.
\end{proof}

\begin{prop}
	Let $T=\{\a_{d+j_1},\dots,\a_{d+j_l}\}$ be a non-empty set  such that \[\{\a_{d+j} \, ; \, \deg(\a_{d+j})< 1\}\subseteq T\subseteq\{\a_{d+j} \, ; \,  \deg(\a_{d+j})\leq 1\}.
	\]
	Then, $J=(\x^{\a_1}-\x^{\a_2},\dots,\x^{\a_1}-\x^{\a_d},\x^{\a_{d+j_1}},\dots,\x^{\a_{d+j_l}})$ is a reduction ideal for $\fm$. 
\end{prop}
\begin{proof}
	Let $n=\max\{e_{1},\dots,e_{s}\}$.  For any $l$ such that $\a_{d+l}\in T$ and any $i=1,\ldots,d$, Lemma~\ref{majic} implies that
	\[
	\x^{n\a_i}=\x^{(n-e_{s})\a_i}\x^{e_{s}\a_i} \in \x^{(n-e_{s})\a_i}J\fm^{e_s-1}\subseteq J\fm^{n-1}.
	\]
	Fix $j$ such that $\a_{d+j} \notin T$. By Lemma~\ref{1-1}, it follows $\x^{n\a_{d+j}}\in (\x^{\a_1},\dots,\x^{\a_d})\fm^{n-1}$ and, then, $\x^{n\a_{d+j}}=\x^{\a_i}\x^\u$, for some $1\leq i\leq d$ and $\x^\u\in\fm^{n-1}$. 
	Therefore,
	\[
	\x^{n^2\a_{d+j}}=\x^{n\a_i}\x^{n\u}\in J\fm^{n-1} \fm^{n^2-n}\subseteq J\fm^{n^2-1}. 
	\] 
	Consequently, it is straightforward to see that every monomial in $\fm^p$ is in $J \fm^{p-1}$, for $p$ large enough.	
\end{proof}

\begin{exmp}
	{\bf 1)} Let $S$ be the semigroup of Example~\ref{exm-mon-red}. We have already showed that $(\x^{\a_1},\x^{\a_2})$ is a minimal reduction ideal of $\fm=(\x^{\a_1},\x^{\a_2},\x^{\a_3})$, but, since $\deg(\a_3)=1$, the previous proposition implies that also $(\x^{\a_1}-\x^{\a_2}, \x^{\a_3})$ is a  minimal reduction ideal of $\fm$.  \\
	{\bf 2)} Let $S$ be the semigroup generated by $\a_1=(2,0)$, $\a_2=(0,4)$, $\a_3=(1,1)$, $\a_4=(1,2)$, and $\a_5=(2,1)$. Then, $\a_1$ and $\a_2$ are the extremal rays of $S$, while $\deg(\a_3)=3/4$, $\deg(\a_4)=1$, and $\deg(\a_5)=5/4$. Therefore, $\fm=(\x^{\a_1},\x^{\a_2},\x^{\a_3},\x^{\a_4},\x^{\a_5})$ has not monomial minimal reductions, but $(\x^{\a_1}-\x^{\a_2},\x^{\a_3})$ is a minimal reduction by the previous proposition.  Note that also $(\x^{\a_1},\x^{\a_2},\x^{\a_3})$ is a reduction of $\fm$, but it is not minimal because $(\x^{\a_1}-\x^{\a_2},\x^{\a_3})\subset (\x^{\a_1},\x^{\a_2},\x^{\a_3})$. 
\end{exmp}

\section{Monomial reduction and multiplicity in dimension two}
 The multiplicity of a projective toric variety can be identified by the Euclidean volume of its convex hull, see for instance \cite[Theorem~4.16]{Sturmfels}.   
In this section,  for affine semigroups $S\subset\NN^2$, we use a similar idea computing  the volume of $P_S$,  to provide an explicit formula for the multiplicity of $R$ in terms of some Ap\'ery sets and an upper bound in terms of the coordinates of the extremal rays. Firstly, the criteria for having monomial minimal reduction ideal are made more concrete for affine semigroups of dimension two. 
Throughout the section, $S$ denotes  an affine semigroup contained in $\NN^2$ minimally generated by $\a_i=(c_{i},d_{i})$ for $i=1,\dots,s+2$, where $\a_1$ and $\a_2$ are the extremal rays.  We recall that every 2-dimensional affine semigroup is simplicial.

 The following result gives another criterion to see whether $\fm$ has a monomial reduction.

\begin{prop} \label{mon red dim 2}
	The following statements are equivalent, where $\lambda=\frac{d_1-d_2}{c_1-c_2}$. 
	\begin{enumerate}
		\item[(a)] $(\x^{\a_1},\x^{\a_2})$ is a  reduction ideal of $\fm$.
		\item[(b)] $d_i\geq\lambda(c_i-c_1)+d_1$, for $i=3,\dots,s+2$.
		\item[(c)] $d_i\geq\lambda(c_i-c_2)+d_2$, for $i=3,\dots,s+2$.
	\end{enumerate}
\end{prop}
\begin{proof}
	Denote by $\DD$ the line passing through the extremal rays $\a_1,\a_2$. By Remark~\ref{D}, the statement (a) is equivalent to have  all $\a_i$ located on or above the line $\DD$ for $i=3,\dots,s+2$. The result follows by the fact that $\DD$ has the equation 
	\begin{equation*}
	\begin{split}
	y &= \frac{d_1-d_2}{c_1-c_2}(x-c_1)+d_1\\
	&= \frac{d_1-d_2}{c_1-c_2}(x-c_2)+d_2. \qedhere
	\end{split}
	\end{equation*}
\end{proof}

\begin{prop}\label{lem}
	If 	$(\x^{\a_1},\x^{\a_2})$ is a  reduction ideal of $\fm$, then $c_i\geq\min\{c_1,c_2\}$ and 
	$d_i\geq\min\{d_1,d_2\}$, for $i=3,\dots,s+2$. 
\end{prop}
\begin{proof}
	Let $\a_i=l_{i_1}\a_1+l_{i_2}\a_2$, where $l_{i_1},l_{i_2}\in\QQ$, for $i=3,\dots,s+2$. By 
	Theorem~\ref{monomial red}, we have $l_{i_1}+l_{i_2}\geq 1$, for $i=3,\dots,s+2$. 
	Thus, the result follows from the inequalities 
	$$\min\{c_1,c_2\}(l_{i_1}+l_{i_2})\leq l_{i_1}c_1+l_{i_2}c_2=c_i$$ and 
	\[ 
	\min\{d_1,d_2\}(l_{i_1}+l_{i_2})\leq l_{i_1}d_1+l_{i_2}d_2=d_i. \qedhere
	\] 	
\end{proof}

As the following example shows, the converse of Proposition~\ref{lem} is not true. 

\begin{exmp}
	Let $S$ be generated by $\a_1=(4,1)$, $\a_2=(1,3)$ and $\a_3=(2,2)$. Then, $c_3=d_3>\min\{c_1,c_2\}=\min\{d_1,d_2\}=1$. On the other hand, $\a_3=\frac{4}{11}\a_1+\frac{6}{11}\a_2$ and $\deg(\a_3)=\frac{10}{11}<1$, thus, $(\x^{\a_1},\x^{\a_2})$ is not a reduction of $\fm$ by Theorem~\ref{monomial red}.
\end{exmp}

 In order to give a concrete upper bound for the multiplicity of $R$  in the Cohen-Macaulay case, we prove a more general result.

\begin{prop} \label{parallelogram}
	The ring $R$ is Cohen-Macaulay if and only if 
	\begin{equation*}
	|\AP(S,E)|+|P_S\setminus\G(S)|=|c_1d_2-c_2d_1|.
	\end{equation*}
\end{prop}
\begin{proof}
	Without loss of generality, we may assume that $c_1d_2> c_2d_1$. The area of the parallelogram spanned by  $\a_1$ and $\a_2$ equals 
	$\det \begin{pmatrix}c_1 & c_2 \\ d_1 & d_2\end{pmatrix}=c_1d_2-c_2d_1$.
	If we define
	\[ e_i=
	\begin{cases}
	d_i \hspace{35pt} &\text{if } c_i= 0 \\
	c_i &\text{if } d_i = 0 \\
	\gcd(c_i,d_i) &\text{otherwise}
	\end{cases}
	\]
	for $i=1,2$, the boundary of that parallelogram contains precisely $2(e_1+e_2)$ lattice points. Then, Pick's theorem (\cite[Theorem 2.8]{BeRo}) implies that $P_{S}$ has $c_1d_2-c_2d_1-(e_1+e_2)+1$ inner lattice points. 
	Given $\a\in P_S$, then $\a$ never equals any of the points 
	$$\{\a_1+\a_2, \a_1+\frac{b_2}{e_2}\a_2, \frac{b_1}{e_1}\a_1+\a_2 \text{, where } 0 \leq b_i < e_i \}$$
	on the boundary of the mentioned parallelogram. Consequently, 
	the points of $P_S$ belong to the union of inner  lattice points and the boundary ones except  $e_1+e_2+1$ points. 
	Therefore, on the one hand Lemma~\ref{lemAp}(a) implies that $|P_S|=|P_S \cap {\rm Gr}(S)|+|P_S \setminus {\rm Gr}(S)|=|r(\AP(S,E))|+|P_S \setminus {\rm Gr}(S)|$, but on the other hand we have
	\begin{eqnarray*}
		|P_S| &=& (c_1d_2-c_2d_1-(e_1+e_2)+1)+2(e_1+e_2)-(e_1+e_2+1)\\
		&=&c_1d_2-c_2d_1.  
	\end{eqnarray*} 
	It is now enough to recall that $R$ is Cohen-Macaulay if and only if $|r(\AP(S,E))|=|\AP(S,E)|$, by Lemma~\ref{lemAp}(b).
\end{proof}

Lemma~\ref{eqm}(c) immediately implies the following:

\begin{cor}\label{e-ub}
	Assume that $R$ is Cohen-Macaulay. The following statements hold:
	\begin{enumerate}
		\item[(a)]	$e(R)\leq |c_1d_2-c_2d_1|$.
		\item[(b)]  $\G(S)=\ZZ^2$ if and only if	
		$|\AP(S,E)|=|c_1d_2-c_2d_1|$.
	\end{enumerate}
\end{cor}

\begin{exmp}
{\bf 1)} Let $S$ be generated by $\a_1=(3,1), \a_2=(4,2)$, and $\a_3=(5,2)$. Since in this case $R$ is Cohen-Macaulay, Corollary~\ref{e-ub} implies that $e(R) \leq 6-4=2$. Moreover, the ring $R$ is not regular because its embedding dimension is equal to three and, consequently, $e(R)=2$. \\ 
{\bf 2)} Let $S$ be the semigroup generated by  $\a_1=(5,3)$, $\a_2=(3,5)$, and $\a_3=(2,2)$. Then, $\AP(S,E)=\{(0,0),(2,2),(4,4),(6,6)\}$ and $P_S\setminus\G(S) = \{(1,1), (3,3), (5,5), \break (7,7), (2,3), (3,2), (3,4), (4,3), (5,4), (4,5), (6,5), (5,6)\}$. Thus,
\[
	|\AP(S,E)|+|P_S\setminus\G(S)|=4+12=25-9=16.
\]
Moreover, a computation by Macaulay2 \cite{Mac2} shows that $e(R)=2$, where $R=\mathbb{Q}[\x^{\a_1},\x^{\a_2},\x^{\a_3}]$. 
\end{exmp}

In the sequel we are looking for the multiplicity of $R$ in terms of $\AP(S,E)$, when $R$ has a monomial minimal reduction ideal but it is not necessarily Cohen-Macaulay. For an integer $k$, we set
	\[
	A_k=\{\w\in\AP(S,E) \, ; \, \w+k\a_2\notin \AP(S,\a_1)\}.
	\]
\begin{prop}\label{3.8}
Let  $I=(\x^{\a_1},\x^{\a_2})$ be  a reduction ideal of $\fm$ and $r=r_I(R)$. Then, $A_k=A_r$ for $k\geq r$ and $e(R)=|\AP(S,E)|-|A_r|$. In particular, $R$ is Cohen-Macaulay precisely when $A_r$ is an empty set.
\end{prop}
\begin{proof} Let $\AP(S,E)=\{\w_1,\dots,\w_t\}$. Then, $\ord_S(\w_i)\leq r$ for $1 \leq i \leq t$ by Remark~\ref{rem:red}. Let
	$$m_j=\min\{m \, ; \, \ord_S(\w_j+m\a_2)\geq r\},$$ for $j=1,\dots,t.$
If $\ord_S(\w_j+m_j\a_2)=r'>r$, then $m_{j}>0$, $\x^{\w_j+m_j\a_2} \in \fm^{r'}=(\x^{\a_1},\x^{\a_2})\fm^{r'-1}$ and, thus, 
 $\w_j+m_j\a_2=\w+\a_i$ for some $\w\in S$ with $\ord_S(\w)=r'-1$ and $1\leq i\leq2$. Since $\ord_S(\w_j+(m_j-1)\a_2)<r$, we get $i=1$. Consequently $\w_j+m_j \a_2\notin\AP(S,\a_1)$. Note that any element in $\AP(S,\a_1)$ is uniquely written as $\w+m\a_2$, for some $\w\in\AP(S,E)$ and $m\in\NN$. Therefore,
\[
\{\w_j+m_j\a_2 \, ; \, j=1,\dots,t\}\cap\AP(S,\a_1)=\{\w\in\AP(S,\a_1) \, ;\, \ord_S(\w)=r \}.
\]
Let $R'=R/(\x^{\a_1})R$  and let $\mathfrak{n}$ be its maximal ideal. Then, $e(R)=e(R')$, by \cite[Proposition~11.1.9]{Huneke-Swanson}. 
Since $(\x^{\a_1},\x^{\a_2})$ is a minimal reduction ideal of $\fm$, we derive that $(\x^{\a_2})R'$ is a minimal reduction ideal of $\mathfrak{n}$ with reduction number smaller than or equal to $r$.  
As $\dim(R')=1$, its multiplicity is equal to $|\fn^k/\fn^{k-1}|$ for all $k \geq r$. In other words,  
\[e(R)=|\{\w\in\AP(S,\a_1) \, ; \, \ord_S(\w)=k\}|,\]
for all $k\geq r$.
It follows that 
\[
\{\w\in\AP(S,\a_1) \, ; \, \ord_S(\w)=k\}=\{\w_j+(m_j+k-r)\a_2 \, ; \, j=1,\dots,t\}\cap\AP(S,\a_1),	
\]
has $e(R)$ elements. Implicitly, this means that $\w_j+m_j\a_2\in\AP(S,\a_1)$ if and only if $\w_j+k\a_2\in\AP(S,\a_1)$ for $k\geq m_j$. Thus, $e(R)=|\AP(S,E)|-|A_k|=|\AP(S,E)|-|A_r|$. 
The last statement follows by Proposition~\ref{prop:mon}.	
\end{proof}

\section{The  associated graded ring }

In this section we explore the properties of the associated graded ring $\gr_\fm(R)=\oplus_{i=0}^\infty \, \fm^i/\fm^{i+1}$ of $R$ assuming that $\fm$ has a monomial minimal reduction ideal.
We first recall a result which will be useful in what follows. There, given $f \in R$, we denote its image in $\gr_\fm(R)$ by $f^*$,  whereas if $J$ is an ideal of $R$, we set $\ini(J)=\{f^{*} \, ; \, f \in J\}$.

\begin{prop}\cite[Proposition~3.8]{Verma}\label{Verma}
	Let $f_1,\dots,f_d\in\fm$. Then, $(f_1,\dots,f_d)$ is  a minimal reduction ideal of $\fm$ if and only if $(f_1^*,\dots,f_d^*)$ is a homogeneous system of parameters of $\gr_\fm(R)$.	
\end{prop}

 Our  first aim is to characterize the Cohen-Macaulay property of $\gr_\fm(R)$. This is a generalization of   \cite[Theorem~7]{G}, where  the case $d=1$ is considered.
 
 \begin{prop}\label{lem:gr}
 If  $(\x^{\a_1},\dots,\x^{\a_d})$ is a reduction ideal of $\fm$, then the following statements are equivalent:
 		\begin{enumerate}
 			\item[(a)] $\gr_\fm(R)$ is a Cohen-Macaulay ring;
 			\item[(b)]$(\x^{\a_1})^*,\dots,(\x^{\a_d})^*$ provide a regular sequence in $\gr_\fm(R)$;
 			\item[(c)]$R$ is Cohen-Macaulay and $(\x^{\a_1},\dots,\x^{\a_d})\cap\fm^n=(\x^{\a_1},\dots,\x^{\a_d})\fm^{n-1}$,  for all $n\geq1$;
 			\item[(d)] $R$ is Cohen-Macaulay and  $(\x^{\a_i})^*$ is a non-zero-divisor in $\gr_\fm(R)$, for $i=1,\dots,d$.
 		\end{enumerate}
  \end{prop}
 \begin{proof}
By  Lemma~\ref{mm}, the monomial ideal $(\x^{\a_1},\dots,\x^{\a_d})$ is a minimal reduction ideal of $\fm$. 
If  $R$ is Cohen-Macaulay, then  the system of parameters   $\x^{\a_1},\dots,\x^{\a_d}$ is a regular sequence in $R$.

 (a) $\Leftrightarrow$ (b) 	
  Proposition~~\ref{Verma} implies that  the image of $\x^{\a_1},\dots,\x^{\a_d}$ in  $\gr_\fm(R)$ is a system of parameters. Therefore, $\gr_\fm(R)$ is Cohen-Macaulay if and only if $(\x^{\a_1})^*,\dots,(\x^{\a_d})^*$ provide a regular sequence in $\gr_\fm(R)$.
  
  (b) $\Leftrightarrow$ (c)
  Since $\depth(\gr_{\fm}(R))\leq\depth(R)$, (b) implies that  $R$ is Cohen-Macaulay. Thus, the system of parameters   $\x^{\a_1},\dots,\x^{\a_d}$ is a regular sequence in $R$. Now, the equivalence between (b) and (c) follows by
 Valabrega-Valla criterion \cite{Valabrega-Valla} (see also \cite[Theorem~5.16]{Verma}). 
 
 (d) $\Rightarrow$ (c) By (d) we have  
 \begin{equation}\label{b}
 \x^{\a_i}(\fm^{t-1}\setminus\fm^t)\subseteq\fm^t\setminus\fm^{t+1},
 \end{equation}		
 for all $t\geq 1$ and $1\leq i\leq d$. Let $\x^\c\in(\x^{\a_1},\dots,\x^{\a_d})\cap\fm^n$, for some $n\geq 1$. Then, $\c=\a_i+\b$ where  $\b\in S$ and $1\leq i\leq d$. Let $\ord_S(\b)=t-1$. Then, $\x^\c=\x^{\a_i+\b}\in\fm^t\setminus\fm^{t+1}$ by (\ref{b}). As $\x^c\in\fm^n$, it follows that $n\leq t$ and, equivalently, $\x^\b\in\fm^{n-1}$. Thus, $\x^c\in(\x^{\a_1},\dots,\x^{\a_d})\fm^{n-1}$ and we get (c). 
 
 (b) $\Rightarrow$ (d) It is clear.
 \end{proof}

 We now prove a lemma that will be useful to show other characterizations of the Cohen-Macaulay property of $\gr_{\fm}(R)$.

\begin{lem}\label{lem:r}
	Let  $r=\max\{\ord_S(\w) \, ; \, \w\in\AP(S,E)\}$. If $\ord_S(\a_i+\b)=\ord_S(\b)+1$, for all and $i=1,\dots,d$ and $\b\in S$ with $\ord_S(\b)< r$, then $r=r_I(R)$. 
\end{lem}
\begin{proof}
	Given $\x^\a\in\fm^{r+1}$, we have $\ord_S(\a) \geq r+1$ and, then, $\a=\a_i+\b$ for some $1\leq i\leq d$ and $\b\in S$. By assumption, it is not possible that $\ord_S(\b)< r$.  Therefore, $\x^\b \in \fm^r$ and, consequently, $\x^\a\in I\fm^r$. This shows that $r_I(R)\leq r$ and the equality follows by Remark~\ref{rem:red}.
\end{proof}

 The following theorem generalizes  some results proved in \cite{G} about numerical semigroup rings.

\begin{thm}\label{thm:gr}
The following statements are equivalent:
\begin{enumerate}
	\item[(a)] $\gr_{\fm}(R)$ is Cohen-Macaulay and  $(\x^{\a_1},\dots,\x^{\a_d})$ is a reduction ideal of $\fm$;
	\item[(b)] $R$ is  Cohen-Macaulay and $\ord_S(\b+\a_i)=\ord_S(\b)+1$, for all $\b\in S$ and  $i=1,\dots,d$;
	\item[(c)] $R$ is Cohen-Macaulay and  $\ord_S(\b+\sum^d_{i=1}n_i\a_i)=\ord_S(\b)+\sum^d_{i=1}n_i$, for all $\b\in\Ap(S,E)$ and $n_1,\dots,n_d\in\NN$.
\end{enumerate}
	In this case, the reduction number is equal to $\max\{\ord_S(\b) \, ; \, \b\in\Ap(S,E)\}$.
\end{thm}
\begin{proof}
(a) $\Rightarrow$ (b) It follows by Proposition~\ref{lem:gr}.

(b) $\Rightarrow$ (c) This is clear.

(c) $\Rightarrow$ (b) Any element $\b\in S$ may be written as $\b=\sum^d_{j=1}n_j\a_j+\b'$ for some $\b'\in\Ap(S,E)$ and $n_1,\dots,n_d\in\NN$. Let $1\leq i\leq d$. Then,
\[
\b+\a_i=\sum^{i-1}_{j=1}n_j\a_j+(1+n_i)\a_i+\sum^{d}_{j=i+1}n_j\a_j+\b'.
\]
Therefore, $\ord_S(\b+\a_i)=1+\sum^d_{j=1}n_j+\ord_S(\b')=1+\ord_S(\b)$.

(b) $\Rightarrow$ (a)
 Let $c_j$ denote the smallest non-zero integer such that 
$c_j\a_{d+j}\in\sum^d_{i=1}\NN \a_i$
for $j=1,\dots,r$.
Then, there are  unique integers  $l_{j_1},\dots,l_{j_d}\in\NN$ such that 
\[
c_j\a_{d+j}=\sum^d_{i=1}l_{j_i}\a_i.
\]
Then, $c_j\leq\ord_S(c_j\a_{d+j})=\ord_S(\sum^d_{i=1}l_{j_i}\a_i)=\sum^d_{i=1}l_{j_i},$
where the last equality holds by (b). Now,  $\deg(\a_{d+j})=(\sum^d_{i=1}l_{j_i})/c_j\geq 1$, which means that  $(\x^{\a_1},\dots,\x^{\a_d})$ is a reduction ideal of $\fm$ by Theorem~\ref{monomial red}. Moreover, $\gr_{\fm}(R)$ is Cohen-Macaulay by the previous lemma.   
The last statement follows by Lemma~\ref{lem:r}. 
\end{proof}

\begin{exmp}\label{ex}
	Let $S$ be generated by $\a_1=(6,0), \a_2=(0,4), \a_3=(3,3)$, and $\a_4=(3,9)$. Then, $\deg(\a_i)\geq1$ for $i=3,4$, which shows that $\fm$ has a monomial minimal reduction ideal, by Theorem~\ref{thm:gr}.
As 	$\Ap(S,E)=\{(0,0),(3,3),(3,9),(6,6)\}$,  $\KK[S]$ is Cohen-Macaulay by Lemma~\ref{lemAp}. However, since the order of $\a_4+\a_1=3\a_3$ is greater than $2$, $\gr_{\fm}(R)$ is not Cohen-Macaulay.
\end{exmp}

Let $I$ be a minimal reduction ideal of $\fm$. Then, the reduction number of $\fm$ with respect to $I$ is at most $e(R)$, see \cite{Trung-1987}.  If $\KK$ has zero characteristic, Vasconcelos showed in \cite{Vasconcelos-1996}, that $r_I(R)\leq e(R)-1$. For  homogeneous affine simplicial semigroups, $I$ is a reduction ideal of $\fm$ and in \cite{Hoa-Stuckrad} it is proved that $r_I(R)\leq e(R)-\codim(R)$.  
In the following we find another class of simplicial affine semigroups satisfying this property.

\begin{prop}\label{HS}
	Let $I=(\x^{\a_1},\dots,\x^{\a_d})$ be a minimal reduction ideal of $\fm$. If $\gr_\fm(R)$ is Cohen-Macaulay, then  $r_I(R)\leq e(R)-\codim(R)$.
\end{prop}
\begin{proof}
	By Theorem~\ref{thm:gr}, we have $\ord_S(\b+\a_i)=\ord_S(\b)+1$ for all $\b\in S$ and $1\leq i\leq d$. Thus, $r_I(R)=\max\{\ord_S(\w) \, ; \, \w\in\AP(S,E)\}$, by Lemma~\ref{lem:r}.  If $r_I(R)=1$, then $R$ is a ring of minimal multiplicity and the claim follows by Proposition \ref{min red and min mult}. Thus, we assume that $r_I(R) \geq 2$.
	Let $\w\in\AP(S,E)$ be such that $r_I(R)=\ord_S(\w)$. Then, there are some elements $\w_i\in S$ such that $\w_i\preceq_S \w$ and $\ord_S(\w_i)=i$ for $i=2,\ldots,r_I(R)$. Since $\{\0,\a_{d+1},\dots,\a_{d+s},\w_2,\dots, \w_{r_I(R)}\}\subseteq\AP(S,E)$  and the elements of the first set are pairwise distinct, we derive $|\AP(S,E)|\geq r_I(R) +\codim(R)$. Now, the result follows by  Proposition~\ref{prop:mon}.
\end{proof}

The Cohen-Macaulay condition of $\gr_\fm(R)$ in Proposition~\ref{HS} can not be removed, even if $R$ is Cohen-Macaulay. 

\begin{exmp}\label{ex:4.7}
	Let $S$ be the affine semigroup generated by 
	$\a_1=(6,0), \a_2=(0,4), \a_3=(3,3), \a_4=(3,9)$. As we have seen in Example~\ref{ex}, $R$ is Cohen-Macaulay but $\gr_\fm(R)$ is not. By Proposition~\ref{prop:mon}, $e(R)=|\Ap(S,E)|=4$. Moreover, since $\x^{3\a_3}\in\fm^3\setminus(\x^{\a_1},\x^{\a_2})\fm^2$, we have $r_I(R)>2=e(R)-2$. 
\end{exmp}

\begin{rem}
	In Proposition~\ref{monomial red} we found another bound for the reduction number of $R$ without assuming that $\gr_\fm(R)$ is Cohen-Macaulay. These two bounds are not comparable in general. For instance, when $R$ is a Cohen-Macaulay ring of minimal multiplicity, we have $r_I(R)=e(R)-\codim(R)=|\Ap(S,E)|-\codim(R)=1$ by Proposition~\ref{min red and min mult}, which can be strictly smaller than $l\codim(R)-1$ obtained in Proposition~\ref{monomial red}. Now, consider the semigroup generated by  $\a_1=(5,0), \a_2=(0,5), \a_3=(6,0), \a_4=(0,6)$. Then, $\KK[S]$ is Cohen-Macaulay and $\fm$ has a monomial reduction. Thus, $e(R)=|\Ap(S,E)|=25$, by Proposition~\ref{prop:mon}, and $l\codim(R)-1=11<23=e(R)-\codim(R)$. 
\end{rem}

 We conclude this section by investigating the Gorenstein property of $\gr_{\fm}(R)$. Define the order $\preceq$  in $\ZZ^d$ by $\a\preceq \b$,  if $\b-\a\in S$.
 Then, $\KK[S]$ is a  Gorenstein ring if and only  if it is Cohen-Macaulay and $\Ap(S,E)$ has a single maximal element with respect to $\preceq$, see \cite[Theorems~4.6 and 4.8]{RG-1998}.

\begin{rem}\label{beta}
Note that $\{\x^{\a_1},\ldots,\x^{\a_{d}}\}$ is a system of parameters for $R$.   Consider the zero-dimensional ring $\bar{R}=R/J$, where $J=(\x^{\a_1},\ldots,\x^{\a_d})$ and let $\bar{\fm}=\fm/J$. Then, $\bar{R}$ has length $\lambda(\bar{R})=|\AP(S,E)|$. 
Since $\ini(J)=J$, we have  
	\[
	\gr_\fm(R)/J=\gr_{\bar{\fm}}(\bar{R})=\bigoplus_{i\geq 0}\bar{G}_i,
	\]
	 where
	$\bar{G}_i=\frac{(\fm/J)^i}{(\fm/J)^{i+1}}$ and $\lambda_R(\bar{G}_i)=|\{\w\in\AP(S,E) \, ; \, \ord_S(\w)=i\}|$. 
	
	In particular, $\bar{G}_i=0$ for all $i>\max\{\ord_S(\w) \, ; \, \w\in\AP(S,E)\}$. We will
	use the following notations to refer to these numbers:
	\begin{equation*}
	\beta_i(S):=\lambda_R(\bar{G}_i)=|\{\w\in\AP(S,E) \, ; \, \ord_S(\w)=i\}|,
	\end{equation*}
	\begin{equation*}
	d(S):=\max\{i\in\NN \, ; \, \bar{G}_i\neq 0\}=\max\{\ord_S(\w) \, ; \, \w\in\AP(S,E)\}.
	\end{equation*}
\end{rem}

 For elements $\a , \b\in S$, we write $\a\preceq_M\b$ if $\a\preceq\b$ and $\ord_S(\b)=\ord_S(\b-\a)+\ord_S(\a)$.
The following lemma is obtained from \cite[Lemma~2.7]{JZ-2014}, replacing the Ap\'{e}ry set with respect to an element $x$, by $\AP(S,E)$. This extends the result  proved by Bryant in \cite{Bryant} for numerical semigroups. 
 
\begin{lem}\label{MG}
	If $R$ is Gorenstein, then the following statements
	are equivalent:
	\begin{enumerate}
		\item[(a)] $\gr_{\bar{\fm}}(\bar{R})$ is Gorenstein;
		\item[(b)] $\beta_i(S)=\beta_{d(S)-i}(S)$ for $0\leq i\leq
		\lfloor\frac{d(S)+1}{2}\rfloor$;
		\item[(c)] $\Max_{\preceq_M} \AP(S,E)$ is a singleton. 
	\end{enumerate}
\end{lem}

\begin{prop}\label{4.7}
Assume that $(\x^{\a_1},\dots,\x^{\a_d})$ is a  reduction ideal of $\fm$. The following statements are equivalent:
	\begin{enumerate}
		\item[(a)] $\gr_{\fm}(R)$ is Gorenstein;
		\item[(b)] $\gr_{\fm}(R)$ is Cohen-Macaulay and  $\Max_{\preceq_M} \AP(S,E)$ has a single element. 
	\end{enumerate}
\end{prop}
\begin{proof}
	In both statements $\gr_{\fm}(R)$ is Cohen-Macaulay. Thus, $(\x^{\a_1})^*,\dots,(\x^{\a_d})^*$ is a regular sequence in $\gr_\fm(R)$,  by Proposition~\ref{lem:gr}.
	Therefore,   $\gr_{\fm}(R)$ is Gorenstein if and only if   $\gr_{\bar{\fm}}(\bar{R})$ is Gorenstein and, hence, the result follows by Lemma~\ref{MG}.
\end{proof}

\begin{exmp}
{\bf 1)} Assume that $S$ is generated by $\a_1=(0,2)$, $\a_2=(2,1)$, $\a_3=(0,3)$, and $\a_4=(1,2)$. Then, $(\x^{\a_1},\x^{\a_2})$ is a minimal reduction of $\fm$ and one can see that $\gr_{\fm}(R)$ is Cohen-Macaulay by using Theorem \ref{thm:gr}. Moreover, $\AP(S,E)=\{(0,0),(0,3),(1,2),(1,5)\}$ and, therefore, $(1,5)$ is the unique maximal element of $\AP(S,E)$ with respect to $\preceq_M$. Hence, $\gr_{\fm}(R)$ is Gorenstein. \\
{\bf 2)} Let $S$ be the semigroup generated by $(3,1)$, $(0,4)$, $(1,3)$, and $(2,2)$, where $(3,1)$ and $(0,4)$ are the extremal rays of $S$. We have $\AP(S,E)=\{(0,0),(1,3),(2,2)\}$ and $|\Max_{\preceq_M} \AP(S,E)|=2$. Therefore, $\gr_{\fm}(R)$ is not Gorenstein, even if it is Cohen-Macaulay.
\end{exmp}

  We close this section by a formula for the Castelnouvo-Mumford regularity of $\gr_\fm(R)$, when $\fm$ has monomial minimal reduction and $\gr_\fm(R)$ is Cohen-Macaulay. 
\begin{prop}\label{p}
	If $\fm$ has a monomial minimal reduction ideal and $\gr_\fm(R)$ is Cohen-Macaulay,  then the Castelnuovo-Mumford regularity of $\gr_\fm(R)$ is 
	$$\reg(\gr_\fm(R))=\max\{\ord_S(\w) \, ; \, \w\in\AP(S,E)\}.$$
\end{prop}
\begin{proof}
By \cite[Corollary~3.5]{Trung-1987}, we have  $\reg(\gr_\fm(R))=r_I(\fm)$, and the result follows by Lemma~\ref{lem:r}  and Theorem~\ref{thm:gr}. 
\end{proof}

 As a consequence, we get another proof of \cite[Corollary~4.9]{Ojeda-Tenorio-2017}. Note that in the statement of \cite[Corollary~4.9]{Ojeda-Tenorio-2017} there is a typo, indeed $\reg(I_S)$ should be $\reg(I_S)-1=\reg(R)$.

\begin{cor}
If $\kk[S]$ is Cohen-Macaulay and the defining ideal $I_S$ of $S$ is homogeneous for the standard grading, the Castelnuovo-Mumford regularity of $R$ is 
	\[
	\reg(R)=\max\{\ord_S(\w) \, ; \, \w\in\AP(S,E)\}.
	\] 
\end{cor}
\begin{proof}
Since $I_S$ is homogeneous, $R\cong\gr_\fm(R)$ and the result follows by Proposition~\ref{p}. 
\end{proof}

\begin{exmp}
Let $S$ be the affine semigroup generated by $\a_1=(5,0,0), \a_2=(0,5,0), \a_3=(0,0,5), \a_4=(4,1,0), \a_5=(2,0,3), \a_6=(1,0,4)$ and $\a_7=(1,3,1)$. As it is shown in \cite[Section~6]{Briales-}, $I_S$ is homogeneous, $\reg(R)=\reg(I_S)-1=6$, and $(18,10,2)=2\a_7+4\a_4$ has the maximum order, $6$, among the elements of $\Ap(S,E)$. 
\end{exmp}

\subsection*{Acknowledgments}
The first author was partially funded by
the project ``Propriet\`{a} locali e globali di anelli e di variet\`{a} algebriche"-PIACERI 2020-2022 Universit\`{a} degli Studi di Catania.
The second author acknowledges the receipt of a grant from the ICTP-INdAM Research in Pairs Programme, Trieste, Italy. She was in part supported by a grant from IPM (No. 1400130112).
The third author was partially supported by INdAM, more precisely he was ``titolare di un Assegno di Ricerca dell'Istituto Nazionale di Alta Matematica''. 
Part of this work has been developed during the visit of the second author to the Department of Mathematics and Computer Science, University of Catania in 2019. The authors would like to thank the Department for its hospitality and support.\\
The authors also thank the anonymous referee for her/his careful reading of this paper and for many remarks that helped them to improve its quality.


\end{document}